\documentclass[11pt]{amsart}

\usepackage{latexsym}
\usepackage{amsmath}
\usepackage{amssymb}

\usepackage{epsfig}

\usepackage{enumerate}

\newtheorem{theorem}{Theorem}[section]
\newtheorem{lemma}[theorem]{Lemma}
\newtheorem{proposition}[theorem]{Proposition}

\newtheorem{sublemma}{}[theorem]

\newcommand{\cN}{{\mathcal N}}
\newcommand{\cS}{{\mathcal S}}
\newcommand{\cT}{{\mathcal T}}

\usepackage{color}
\newcommand{\blue}{\textcolor{black}}

\begin{document}

\title{Display Sets of Normal and Tree-Child Networks}

\thanks{The second and third authors were supported by the New Zealand Marsden Fund.}

\author{Janosch D\"{o}cker}
\address{Department of Computer Science, University of T\"{u}bingen, T\"{u}bingen, Germany}
\email{janosch.doecker@uni-tuebingen.de}

\author{Simone Linz}
\address{School of Computer Science, University of Auckland, Auckland, New Zealand}
\email{s.linz@auckland.ac.nz}

\author{Charles Semple}
\address{School of Mathematics and Statistics, University of Canterbury, Christchurch, New Zealand}
\email{charles.semple@canterbury.ac.nz}



\date{\today}

\begin{abstract}
Phylogenetic trees canonically arise as embeddings of phylogenetic networks. We recently showed that the problem of deciding if two phylogenetic networks embed the same sets of phylogenetic trees is computationally hard, \blue{in particular, we showed it to be $\Pi^P_2$-complete}. In this paper, we establish a polynomial-time algorithm for this decision problem if the initial two networks consists of a normal network and a tree-child network. The running time of the algorithm is quadratic in the size of the leaf sets.
\end{abstract}

\maketitle

\section{Introduction}

\blue{Phylogenetic networks rather than phylogenetic (evolutionary) trees provide a more faithful representation of the ancestral history of certain collections of extant species. The reason for this is the existence of non-treelike (reticulate) evolutionary processes such as lateral gene transfer and hybridisation. However, while at the species-level evolution is not necessarily treelike, at the level of genes, we typically assume treelike evolution. Consequently, as phylogenetic networks are frequently viewed as an amalgamation of the ancestral history of genes, we are interested in the phylogenetic trees embedded (displayed) in a given phylogenetic network. From this viewpoint, there has been a variety of studies including the small maximum parsimony problem for phylogenetic networks~\cite{nak05}, deciding if a phylogenetic network is (uniquely) determined by the phylogenetic trees it embeds~\cite{gam12, wil11}, counting the number of phylogenetic trees displayed by a phylogenetic network~\cite{lin13}, and determining if a phylogenetic network embeds a phylogenetic tree more than once~\cite{cor14}.
In this context, one of the most well-known studied computational problems is {\sc Tree-Containment}. Here, the problem is deciding whether or not a given phylogenetic tree is embedded in a given phylogenetic network. In general, the problem is NP-complete~\cite{kan08}, but it has been shown to be decidable in polynomial-time for several prominent classes of phylogenetic networks~\cite{bor16a, gun17, ier10}.}

\blue{Recently posed in~\cite{gun17} for reticulation-visible networks, in this paper we study a natural variation of {\sc Tree-Containment}. In particular, we consider the problem of deciding whether or not two given phylogenetic networks embed the same set of phylogenetic trees. Called {\sc Display-Set-Equivalence}, we recently showed that, in general, this problem is $\Pi^P_2$-complete~\cite{doe18}, that is, complete for the second-level of the polynomial hierarchy. Problems on the second level of the hierarchy are computationally more difficult than problems on the first level which include all NP- and co-NP-complete problems. For further details, see~\cite{sto76}. In contrast, the main result of this paper shows that there is a polynomial-time algorithm for {\sc Display-Set-Equivalence} if the given networks are normal and tree-child. The rest of the introduction formally defines {\sc Display-Set-Equivalence}, states the main result, and provides additional details.}

A {\em phylogenetic network $\cN$ on $X$} is a rooted acyclic directed graph with no arcs in parallel and satisfying the following properties:
\begin{enumerate}[(i)]
\item the (unique) root has out-degree two;

\item a vertex with out-degree zero has in-degree one, and the set of vertices with out-degree zero is $X$; and

\item all other vertices have either in-degree one and out-degree two, or in-degree two and out-degree one.
\end{enumerate}
For technical reasons, if $|X|=1$, we additionally allow a single vertex labelled by the element in $X$ to be a phylogenetic network. The vertices in $\cN$ of out-degree zero are called {\em leaves}, and so $X$ is referred to as the {\em leaf set} of $\cN$. Furthermore, vertices of in-degree one and out-degree two are {\em tree vertices}, while vertices of in-degree two and out-degree one are {\em reticulations}. The arcs directed into a reticulation are {\em reticulation arcs}. A {\em phylogenetic $X$-tree} is a phylogenetic network with no reticulations. \blue{Note that, in the literature, what we have called a phylogenetic network is sometimes referred to as a binary phylogenetic network.}

Let $\cN$ be a phylogenetic network. A reticulation arc $(u, v)$ of $\cN$ is a {\em shortcut} if there is a directed path in $\cN$ from $u$ to $v$ that does not traverse $(u, v)$. We say that $\cN$ is a {\em tree-child network} if every non-leaf vertex is the parent of a tree vertex or a leaf. If, in addition, $\cN$ has no shortcuts, then $\cN$ is {\em normal}. To illustrate, in Fig.~\ref{example2}(i), $\cN$ is a tree-child network but it is not normal as the arc $(u, v)$ is a shortcut. As with all other figures in the paper, arcs are directed down the page.

\begin{figure}
\center
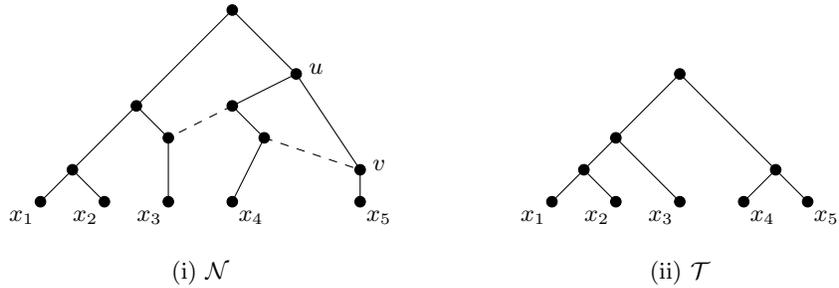
\caption{(i) A tree-child network $\cN$ on $\{x_1, x_2, x_3, x_4, x_5\}$ and (ii) a phylogenetic tree $\cT$ displayed by $\cN$.}
\label{example2}
\end{figure}

Now let $\cN$ be a phylogenetic network on $X$ and let $\cT$ be a phylogenetic $X$-tree. Then $\cN$ {\em displays} $\cT$ if, up to suppressing vertices of in-degree one and out-degree one, $\cT$ can be obtained from $\cN$ by deleting arcs and non-root vertices, in which case, the resulting acyclic directed graph is an {\em embedding} of $\cT$ in $\cN$. In Fig.~\ref{example2}, $\cN$ displays $\cT$, where an embedding of $\cT$ in $\cN$ is shown as solid arcs. Note that there is one other distinct embedding of $\cT$ in $\cN$. Suppose $\cS$ is an embedding of $\cT$ in $\cN$. If $(u, v)$ is an arc in $\cN$, we say $\cS$ {\em uses} $(u, v)$ if $(u, v)$ is an arc in $\cS$; otherwise, $\cS$ {\em avoids} $(u, v)$. Furthermore, note that the root of $\cS$ is the root of $\cN$, and so the former may have out-degree one. The set of phylogenetic $X$-trees displayed by $\cN$, called the {\em display set} of $\cN$, is denoted by $T(\cN)$.

The problem of interest in this paper is the following decision problem:

\noindent {\sc Display-Set-Equivalence} \\
\noindent {\bf Input.} Two phylogenetic networks $\cN$ and $\cN'$ on $X$. \\
\noindent {\bf Output.} Is $T(\cN)=T(\cN')$?

\noindent It is shown in~\cite{doe18} that, in general, {\sc Display-Set-Equivalence} is $\Pi^P_2$-complete. In contrast, the main result of this paper shows that this decision problem is solvable in polynomial time if $\cN$ is normal and $\cN'$ is tree-child. In particular, we have

\begin{theorem}
Let $\cN$ and $\cN'$ be normal and tree-child networks on $X$, respectively. Then deciding if $T(\cN)=T(\cN')$ can be done in time quadratic in the size of $X$.
\label{main}
\end{theorem}

Before continuing, we add some remarks. The proof of Theorem~\ref{main} turned out to be much longer than we originally anticipated. If $\cN'$ has no shortcuts, that is, $\cN'$ is normal, then $T(\cN)=T(\cN')$ if and only if $\cN$ is isomorphic to $\cN'$~\cite{wil11}. However, if $\cN'$ is allowed to have shortcuts, then it is possible for $T(\cN)=T(\cN')$, but $\cN$ is not isomorphic to $\cN'$. \blue{For example, consider the normal and tree-child networks $\cN$ and $\cN'$, respectively, shown in Fig.~\ref{example}. Clearly, $\cN$ is not isomorphic to $\cN'$, but it is easily checked that $T(\cN)=T(\cN')$. While we already knew of instances like that shown in Fig.~\ref{example},} we didn't expect the allowance of shortcuts to raise as many hurdles as it did. \blue{As an aside, we note a related result concerning level-$1$ networks. A phylogenetic network is {\em level-$1$} if no vertex is in two distinct underlying cycles. It is shown in~\cite{gam12} that if $\cN$ and $\cN'$ are level-$1$ networks on $X$ with no underlying cycles of at most four vertices, then $T(\cN)=T(\cN')$ if and only if $\cN$ is isomorphic to $\cN'$.}

\begin{figure}
\center
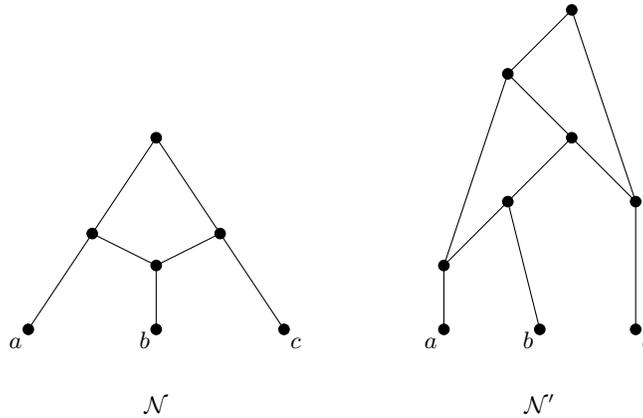
\caption{A normal network $\cN$ and a tree-child network $\cN'$. Now $T(\cN)=T(\cN')$, but $\cN$ is not isomorphic to $\cN'$.}
\label{example}
\end{figure}

\blue{As mentioned above, {\sc Display-Set-Equivalence} was recently posed for when $\cN$ and $\cN'$ are both reticulation-visible. For the reader familiar with the notion of `visibility' (see Section~\ref{prelim}), a phylogenetic network is {\em reticulation-visible} if every reticulation is visible. Tree-child networks are a special subclass of reticulation-visible networks. In particular, tree-child networks are precisely the class of networks in which every vertex is visible~\cite{car09}. Knowing the hurdles to resolve in establishing Theorem~\ref{main}, we pose the apparent simpler problem of determining the complexity of {\sc Display-Set-Equivalence} when $\cN$ and $\cN'$ are both tree-child.}

The paper is organised as follows. Section~\ref{prelim} contains some additional concepts as well as several lemmas concerning tree-child networks. The proof of Theorem~\ref{main} is algorithmic and relies on comparing the structures of $\cN$ and $\cN'$ local to a common pair of leaves. Section~\ref{structure} establishes the necessary structural results to make these comparisons. Depending on the outcomes of the comparisons, the algorithm recurses in one of three ways. The lemmas associated with these recursions are given in Section~\ref{recur}. The algorithm, its correctness, and its running time, and thus the proof of Theorem~\ref{main}, are given in the last section. A more detailed overview of the algorithm underlying the proof of Theorem~\ref{main} is given at the end of the next section.

\section{Preliminaries}
\label{prelim}

Throughout the paper, $X$ denotes a non-empty finite set and all paths are directed. \blue{Furthermore, if $D$ is a set and $b$ is an element, we write $D\cup b$ for $D\cup \{b\}$ and $D-b$ for $D-\{b\}$.}

\noindent {\bf Cluster and visibility sets.} Let $\cN$ be a phylogenetic network on $X$ with root $\rho$, and let $u$ be a vertex of $\cN$. A vertex $v$ is {\em reachable} from $u$ if there is a path from $u$ to $v$. The set of leaves reachable from $u$, denoted $C_u$, is the {\em cluster (set) of $u$}. Furthermore, we say that $u$ is {\em visible} if there is a leaf, $x$ say, such that every path from $\rho$ to $x$ traverses $u$, in which case, {\em $x$ verifies the visibility of $u$}. The set of leaves verifying the visibility of $u$, denoted $V_u$, is the {\em visibility set of $u$}. Note that the visibility set of $u$ is a subset of the cluster set of $u$.

Let $\cT$ be a phylogenetic $X$-tree. A subset $C$ of $X$ is a {\em cluster of $\cT$} if there is a vertex $u$ in $\cT$ such that $C=C_u$. For disjoint subsets $Y$ and $Z$ of $X$, we say that $\{Y, Z\}$ is a {\em generalised cherry} of $\cT$ if $Y$, $Z$, and $Y\cup Z$ are all clusters of $\cT$.

\noindent {\bf Normal and tree-child networks.} Let $u$ be a vertex of a phylogenetic network $\cN$ on $X$. A path $P$ starting at $u$ and ending at a leaf is a {\em tree-path} if every non-terminal vertex is a tree vertex, in which case, $P$ is a {\em tree-path for $u$}. The next lemma is freely-used throughout the paper. \blue{Part~(ii)} is well-known and follows immediately from the definition of a tree-child network, \blue{and (iii) was noted in the introduction}.

\begin{lemma}
\blue{Let $\cN$ be a phylogenetic network. Then the following statements are equivalent:}
\begin{enumerate}[{\rm (i)}]
\item $\cN$ is tree-child,

\item every vertex of $\cN$ has a tree-path, and

\item \blue{every vertex of $\cN$ is visible.}
\end{enumerate}
\label{tree-child}
\end{lemma}
\noindent It \blue{follows from Lemma~\ref{tree-child} that} all verifying sets of $\cN$ are non-empty.

Let $a$ and $b$ be distinct leaves of a phylogenetic network $\cN$, and let $p_a$ and $p_b$ denote the parents of $a$ and $b$, respectively. Then $\{a, b\}$ is a {\em cherry} if $p_a=p_b$. Furthermore, $\{a, b\}$ is a {\em reticulated cherry} if the parent of one of the leaves, say $b$, is a reticulation and $(p_a, p_b)$ is an arc in $\cN$. Note that, if this holds, then $p_a$ is a tree vertex. \blue{The arc $(p_a, p_b)$ is the {\em reticulation arc} of the reticulated cherry $\{a, b\}$.} As with the previous lemma, the next lemma~\cite{bor16b} is freely-used throughout the paper.

\begin{lemma}
Let $\cN$ be a tree-child network on $X$, where $|X|\ge 2$. Then $\cN$ has either a cherry or a reticulated cherry.
\label{cherries}
\end{lemma}

The next lemma is established in~\cite{wil10}.

\begin{lemma}
Let $\cN$ be a normal network on $X$, and let $t$ and $u$ be vertices in $\cN$. Then $C_u\subseteq C_t$ if and only if $u$ is reachable from $t$.
\label{cluster1}
\end{lemma}

Let $\cN$ be a phylogenetic network on $X$. Let $\mathcal S$ be an embedding in $\cN$ of a phylogenetic $X$-tree $\cT$ and let $C$ be a cluster of $\cT$. \blue{Analogous to cluster sets of $\cN$, each vertex $w$ of $\cS$ has a {\em cluster set} and this set consists of the elements in $X$ at the end of a path in $\cS$ starting at $w$. Of course, the cluster set of $w$ relative to $\cS$ is a subset of the cluster set of $w$ relative to $\cN$. The vertex in $\mathcal S$ {\em corresponding} to $C$ is the (unique) vertex $u$ whose cluster set relative to $\cS$ is $C$ and with the property that every other vertex with cluster set $C$ in $\cS$ is on a path from the root of $\cS$ to $u$.}

\begin{lemma}
Let $\cN$ be a normal network on $X$ and let $u$ be a tree vertex of $\cN$. Let $\cT$ be a phylogenetic $X$-tree having cluster $C_u$. If $\mathcal S$ is an embedding of $\cT$ in $\cN$, then the vertex in $\mathcal S$ corresponding to $C_u$ is $u$.
\label{nosplitting}
\end{lemma}

\begin{proof}
Suppose that $\mathcal S$ is an embedding of $\cT$ in $\cN$. Let $t$ be the vertex in $\mathcal S$ corresponding to $C_u$, and observe that $t$ is a tree vertex. Clearly, $C_u\subseteq C_t$ and so, by Lemma~\ref{cluster1}, $u$ is reachable from $t$ on a path $P$ in $\cN$. If $t\neq u$, then, as $\cN$ is normal and therefore has no shortcuts, $t$ is the parent of a vertex, $v$ say, that is not on $P$. Now, there is a tree-path from $v$ to a leaf $\ell$. By construction, $\ell\not\in C_u$. In turn, regardless of whether or not $v$ is a reticulation, this implies that the cluster in $\mathcal S$ corresponding to $t$ contains $\ell$; a contradiction. Thus $t=u$, thereby completing the proof of the lemma.
\end{proof}

\noindent {\bf Deleting arcs and leaves.} Let $\cN$ be a phylogenetic network on $X$, and let $(u, v)$ be an arc of $\cN$. We denote the directed graph obtained from $\cN$ by deleting $(u, v)$ and suppressing any resulting vertices with in-degree one and out-degree one by $\cN\backslash (u, v)$. If $u$ is the root of $\cN$, we additionally delete $u$ (and its incident arc) after deleting $(u, v)$. Moreover, if $b$ is a leaf of $\cN$, then the directed graph obtained from $\cN$ by deleting $b$ (and its incident arc), and suppressing any resulting vertex of in-degree one and out-degree one is denoted by $\cN\backslash b$. Again, if the parent of $b$ is the root of $\cN$, we additionally delete the root (and its incident arc) after deleting $b$.

\blue{Deleting an arc or a leaf of a phylogenetic network doesn't necessarily result in another phylogenetic network. The next lemma, which is also freely used in the paper, gives some sufficient conditions for when these operations result in a phylogenetic network. The proof of this lemma for tree-child networks is established in~\cite{bor16b}. With that in hand, the proof of the lemma for normal networks is straightforward and omitted.}

\begin{lemma}
Let $\cN$ be a tree-child network on $X$.
\begin{enumerate}[{\rm (i)}]
\item \blue{Suppose $\{a, b\}$ is a cherry of $\cN$. Then $\cN\backslash b$ is a tree-child network on $X-b$. Moreover, if $\cN$ is normal, then $\cN\backslash b$ is normal.}

\item \blue{Suppose $(u, v)$ is the reticulation arc of a reticulated cherry. Then $\cN\backslash (u, v)$ is a tree-child network on $X$. Moreover, if $\cN$ is normal, then $\cN\backslash (u, v)$ is normal.}
\end{enumerate}
\label{deleting}
\end{lemma}

Let $\cN$ be a phylogenetic network. A shortcut $(u, v)$ in $\cN$ is {\em trivial} if the parent of $v$ that is not $u$ is a child of $u$.

\begin{lemma}
\blue{Let $\cN$ be a tree-child network} and let $(u, v)$ be a shortcut of $\cN$. \blue{Then $\cN\backslash (u, v)$ is tree-child. Furthermore, if $(u, v)$ is trivial, then $T(\cN)=T(\cN\backslash (u, v))$.}
\label{shortshort}
\end{lemma}

\begin{proof}
\blue{Let $u'$ be the parent of $v$ that is not $u$. Since $\cN$ is tree-child, $u'$ is a tree vertex, and the unique child of $v$ is either a tree vertex or a leaf. It follows that $\cN\backslash (u, v)$ has no parallel arcs, and so $\cN\backslash (u, v)$ is a phylogenetic network. Furthermore, if $w$ is an arbitrary vertex of $\cN$, then any tree-path for $w$ does not traverse $(u, v)$ and so, as $\cN$ is tree-child, $\cN\backslash (u, v)$ is also tree-child.}

\blue{Now, regardless of whether $(u, v)$ is trivial, $T(\cN\backslash (u, v))\subseteq T(\cN)$. So assume that $(u, v)$ is trivial, in which case $(u, u')$ is an arc in $\cN$, and let $\cT$ be a phylogenetic tree displayed by $\cN$. Let $\cS$ be an embedding of $\cT$ in $\cN$. If $\cS$ avoids $(u, v)$, then it is clear that $\cN\backslash (u, v)$ displays $\cT$. On the other hand, if $\cS$ uses $(u, v)$, then by replacing $(u, v)$ with $(u', v)$ we obtain an embedding of $\cT$ in $\cN$ avoiding $(u, v)$, and so $\cN\backslash (u, v)$ displays $\cT$. Note that, as $\cN$ is tree-child, $\cS$ uses $(u, u')$. Hence $T(\cN)\subseteq T(\cN\backslash (u, v))$.}
\end{proof}

We end this section by briefly outlining the algorithm associated with the proof of Theorem~\ref{main}. Called {\sc SameDisplaySet}, the algorithm takes as its input normal and tree-child networks $\cN$ and $\cN'$, respectively, and proceeds by first finding a cherry or a reticulated cherry, $\{a, b\}$ say, in $\cN$. It then considers the structure of $\cN'$ (and if necessary $\cN$) local to leaves $a$ and $b$, and decides whether to return $T(\cN)\neq T(\cN')$ or to continue. This decision is based on three Propositions~\ref{structure1}, \ref{structure2}, and~\ref{structure3}. These propositions give necessary structural properties if $T(\cN)=T(\cN')$. If the algorithm continues, it deletes certain arcs and leaves in $\cN$ and $\cN'$. Lemma~\ref{algorithm1}--\ref{algorithm3} show that the resulting normal and tree-child networks after the deletions, $\cN_1$ and $\cN'_1$ say, display the same set of phylogenetic trees, that is $T(\cN_1)=T(\cN'_1)$, if and only if $T(\cN)=T(\cN')$. The algorithm now recurses on $\cN_1$ and $\cN'_1$ by finding a cherry or a reticulated cherry of $\cN_1$. Eventually, {\sc SameDisplaySet} either stops and returns $T(\cN)\neq T(\cN')$ or it reduces $\cN$ and $\cN'$ to a phylogenetic network consisting of two leaves, in which case $T(\cN)=T(\cN')$.

\section{Structural Properties}
\label{structure}

The purpose of this section is to establish three structural results, namely, Propositions~\ref{structure1}, \ref{structure2}, and~\ref{structure3}. Relative to either a cherry or a reticulated cherry, $\{a, b\}$ say, of $\cN$, these results concern the structures of $\cN$ and $\cN'$ local to $a$ and $b$ if $T(\cN)=T(\cN')$. The first considers when $\{a, b\}$ is a cherry of $\cN$, while the second and third considers when $\{a, b\}$ is a reticulated cherry of $\cN$ and where the parent of $b$ in $\cN'$ is either a tree vertex or a reticulation.

\begin{proposition}
Let $\cN$ and $\cN'$ be normal and tree-child networks on $X$, respectively, and suppose $\cN'$ has no trivial shortcuts. Let $\{a, b\}$ be a cherry of $\cN$. Then $T(\cN)=T(\cN')$ only if $\{a, b\}$ is a cherry of $\cN'$.
\label{structure1}
\end{proposition}

\begin{proof}
Suppose $T(\cN)=T(\cN')$. Note that $\{a, b\}$ is a cherry of every phylogenetic $X$-tree displayed by $\cN$. Let $p'_a$ and $p'_b$ denote the parents of $a$ and $b$ in $\cN'$, respectively. First assume that $p'_b$ is a tree vertex. Then, as $T(\cN)=T(\cN')$, it follows that $C_{p'_b}=\{a, b\}$. Thus the child vertex of $p'_b$ in $\cN'$ that is not $b$ is either $a$ or $p'_a$. In particular, $\{a, b\}$ is either a cherry or a reticulated cherry with reticulation leaf $a$ in $\cN'$. Consider the latter. If $q'$ denotes the parent of $p'_a$ that is not $p'_b$ in $\cN'$, then $(q', p'_a)$ is a shortcut. Otherwise, there is a tree-path from $q'$ to a leaf that is not $b$, and so, using $(q', p'_a)$, it follows that $\cN'$ displays a phylogenetic $X$-tree in which $\{a, b\}$ is not a cherry. Now let $t'$ denote the child vertex of $q'$ that is not $p'_a$. Since $\cN'$ is tree-child, $t'$ is a tree vertex. If \blue{$C_{t'}-a\neq \{b\}$}, then, using $(q', p'_a)$, we have that $\cN'$ displays a phylogenetic $X$-tree not displayed by $\cN$, so \blue{$C_{t'}-a=\{b\}$} and, in particular, $t'=p'_b$. Thus $\{q', p'_a\}$ is a trivial shortcut. Therefore if $p'_b$ is a tree vertex, then $\{a, b\}$ is a cherry of $\cN'$. If $p'_b$ is a reticulation in $\cN'$, then a similar argument leads to the conclusion that $\cN'$ has a trivial shortcut. This completes the proof of the proposition.
\end{proof}

We next consider the relative structure local to leaves $a$ and $b$ in $\cN$, where $\{a, b\}$ is a reticulated cherry of $\cN$. For the next three results, we suppose that $\{a, b\}$ is a reticulated cherry of $\cN$ as shown in Fig.~\ref{cherry1}. Note that, although not shown, if $C_q-(V_q\cup b)$ is nonempty, then $\cN$ contains paths from the root $\rho$ to leaves in $C_q-(V_q\cup b)$ avoiding $q$.

\begin{figure}
\center
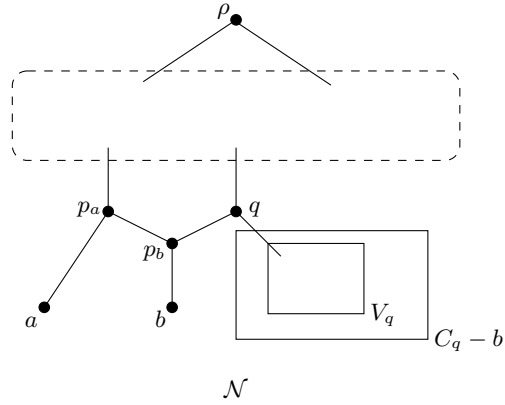
\caption{The structure of $\cN$ local to the reticulated cherry $\{a, b\}$. Note that, for each leaf $\ell\in C_q-(V_q\cup b)$, there is a path from $\rho$ to $\ell$ avoiding $q$.}
\label{cherry1}
\end{figure}

The proof of the next lemma is straightforward and omitted.

\begin{lemma}
Let $\cN$ be a normal network on $X$, and suppose that $\{a, b\}$ is a reticulated cherry of $\cN$ as shown in Fig.~\ref{cherry1}. If $\cT$ is a phylogenetic $X$-tree displayed by $\cN$, then either $\{a, b\}$ is a cherry or $\{b, C'_q\}$ is a generalised cherry of $\cT$, where $V_q\subseteq C'_q\subseteq C_q-b$ \blue{and $a\not\in C_q$}. Moreover, $\cN$ displays phylogenetic $X$-trees in which $\{a, b\}$, $\{b, V_q\}$ and $\{b, C_q-b\}$ are generalised cherries.
\label{cherries2}
\end{lemma}

\begin{figure}
\center
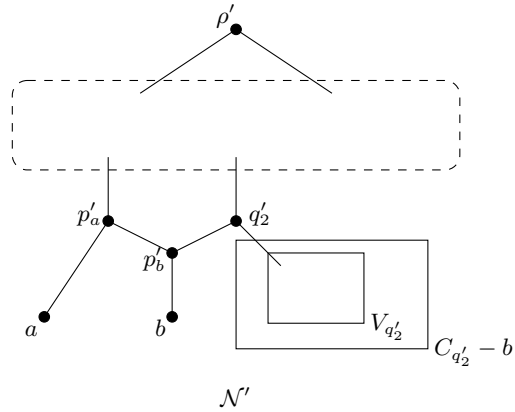
\caption{The structure of $\cN'$ local to the leaves $a$ and $b$, when $\{a, b\}$ is the reticulated cherry of $\cN$ as shown in Fig.~\ref{cherry1} and the parent of $b$ in $\cN'$ is a reticulation. Here, $V_{q'_2}=V_q$ and $C_{q'_2}=C_q$. Note that, for each leaf $\ell\in C_{q'_2}-(V_{q'_2}\cup b)$, there is a path from $\rho'$ to $\ell$ avoiding $q'_2$.}
\label{cherry2}
\end{figure}

\begin{proposition}
Let $\cN$ and $\cN'$ be normal and tree-child networks on $X$, respectively, and suppose that $\{a, b\}$ is a reticulated cherry of $\cN$ as shown in Fig.~\ref{cherry1}. If the parent of $b$ in $\cN'$ is a reticulation, then $T(\cN)=T(\cN')$ only if, up to isomorphism, $\{a, b\}$ is a reticulated cherry of $\cN'$ as shown in Fig.~\ref{cherry2}, where $V_{q'_2}=V_q$ and $C_{q'_2}=C_q$.
\label{structure2}
\end{proposition}

\begin{proof}
Let $\{a, b\}$ be a reticulated cherry of $\cN$ as shown in Fig.~\ref{cherry1}. Thus $p_a$ and $p_b$ denote the parents of $a$ and $b$ in $\cN$, respectively, where $p_b$ is a reticulation, and $q$ denotes the parent of $p_b$ in $\cN$ that is not $p_a$. Since $\cN$ is normal, $(q, p_b)$ is not a shortcut and $a\not\in C_q$. Suppose $T(\cN)=T(\cN')$, and consider $\cN'$. Let $p'_a$ and $p'_b$ denote the parents of $a$ and $b$ in $\cN'$, respectively, where $p'_b$ is a reticulation. Let $q'_1$ and $q'_2$ denote the parents of $p'_b$ in $\cN'$.

\begin{sublemma}
Neither $(q'_1, p'_b)$ nor $(q'_2, p'_b)$ is a shortcut.
\label{noshortcut}
\end{sublemma}

\begin{proof}
Assume at least one of $(q'_1, p'_b)$ and $(q'_2, p'_b)$ is a shortcut. Without loss of generality, we may assume $(q'_2, p'_b)$ is a shortcut, and so $(q'_1, p'_b)$ is not a shortcut. \blue{By Lemma~\ref{cherries2}, $\cN'$ displays a phylogenetic $X$-tree with $\{a, b\}$ as a cherry as well as a phylogenetic $X$-tree with $\{b, V_q\}$ as a generalised cherry. Thus $V_q\cup a\subseteq C_{q'_2}$.} But then, using $(q'_2, p'_b)$, we have that $\cN$ displays a phylogenetic $X$-tree with a generalised cherry $\{b, Z\}$, where $|Z|\ge 2$ and $a\in Z$, contradicting Lemma~\ref{cherries2}. Hence neither $(q'_1, p'_b)$ nor $(q'_2, p'_b)$ is a shortcut.
\end{proof}

By (\ref{noshortcut}), neither $(q'_1, p'_b)$ nor $(q'_2, p'_b)$ is a shortcut. Therefore, for some $i\in \{1, 2\}$, we have $C_{q'_i}-b=\{a\}$. If not, then either there is no phylogenetic $X$-tree displayed by $\cN'$ with $\{a, b\}$ as a cherry, or there is a phylogenetic $X$-tree displayed by $\cN'$ in which $\{b, Z\}$ is a generalised cherry, where $|Z|\ge 2$ and $a\in Z$, contradicting Lemma~\ref{cherries2}. Without loss of generality, we may assume that $C_{q'_1}-b=\{a\}$ and so, as $\cN'$ is tree-child, $q'_1=p'_a$. That is, $\{a, b\}$ is a reticulated cherry of $\cN'$. Observe that $a\not\in C_{q'_2}-b$.

By Lemma~\ref{cherries2}, $\cN$ displays a phylogenetic $X$-tree with generalised cherry $\{b, V_q\}$ and so, as $T(\cN)=T(\cN')$, it follows that $V_q\subseteq C_{q'_2}-b$ and $V_{q'_2}\subseteq V_q$. In turn, as $\cN'$ displays a phylogenetic $X$-tree with generalised cherry $\{b, V_{q'_2}\}$, we have $V_{q'_2}\subseteq C_q-b$ and $V_q\subseteq V_{q'_2}$. Thus $V_q=V_{q'_2}$. Furthermore, as $\cN$ displays a phylogenetic $X$-tree with generalised cherry $\{b, C_q-b\}$, and $\cN'$ displays a phylogenetic $X$-tree with generalised cherry $\{b, C_{q'_2}-b\}$, we deduce that $C_q-b\subseteq C_{q'_2}-b$ and $C_{q'_2}-b\subseteq C_q-b$, so $C_q-b=C_{q'_2}-b$. Thus $\{a, b\}$ is a reticulated cherry of $\cN'$ as shown in Fig.~\ref{cherry2}, and this completes the proof of the proposition.
\end{proof}

\begin{figure}
\center
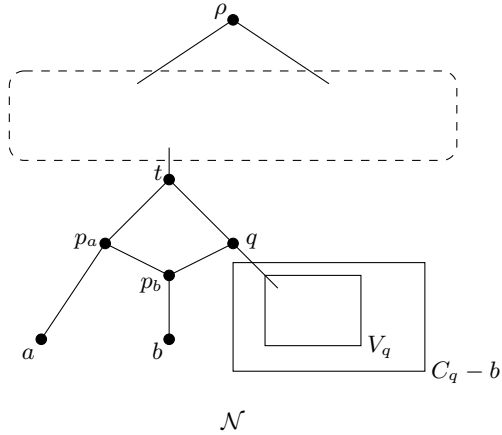
\caption{The structure of $\cN$ local to the leaves $a$ and $b$, when $\{a, b\}$ is a reticulated cherry of $\cN$ as shown in Fig.~\ref{cherry1} and the parent of $b$ in $\cN'$ is a tree vertex. Note that, for each leaf $\ell\in C_q-(V_q-b)$, there is a path from $\rho$ to $\ell$ avoiding $q$.}
\label{cherry3}
\end{figure}

\begin{figure}
\center
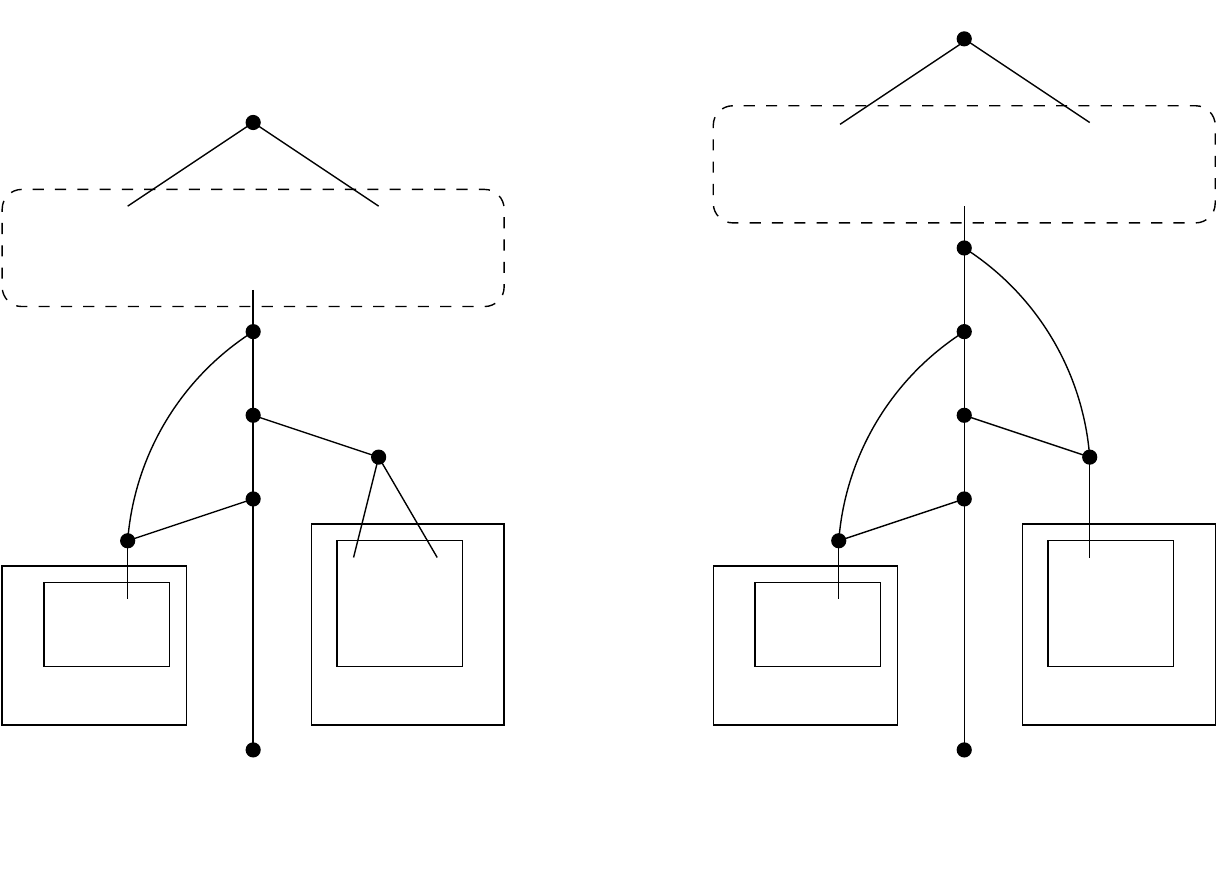
\caption{The two possible structures of $\cN'$ local to the leaves $a$ and $b$, when $\{a, b\}$ is a reticulated cherry of $\cN$ as shown in Fig.~\ref{cherry1} and the parent of $b$ in $\cN'$ is a tree vertex, where $\{V_{v'_1}, V_{v'_2}\}=\{\{a\}, V_q\}$ and $\{C_{v'_1}, C_{v'_2}\}=\{\{a\}, C_q-b\}$. Note that, if $C_{v'_i}\neq \{a\}$, then, for each leaf $\ell\in C_{v'_i}-V_{v'_i}$, there is a path from $\rho'$ to $\ell$ avoiding $v'_i$. Furthermore, in (a), $v'_2$ could be a leaf.}
\label{noncherry}
\end{figure}

\begin{proposition}
Let $\cN$ and $\cN'$ be normal and tree-child networks on $X$, respectively, and suppose that $\cN'$ has no trivial shortcuts and $\{a, b\}$ is a reticulated cherry of $\cN$ as shown in Fig.~\ref{cherry1}. If the parent of $b$ in $\cN'$ is a tree vertex, then $T(\cN)=T(\cN')$ only if, up to isomorphism, in $\cN$, leaves $a$ and $b$ are as shown in Fig.~\ref{cherry3} and, in $\cN'$, leaves $a$ and $b$ are as shown in either Fig.~\ref{noncherry}(a) or Fig.~\ref{noncherry}(b), where $\{V_{v'_1}, V_{v'_2}\}=\{\{a\}, V_q\}$ and $\{C_{v'_1}, C_{v'_2}\}=\{\{a\}, C_q-b\}$.
\label{structure3}
\end{proposition}

\begin{proof}
Let $\{a, b\}$ be a reticulated cherry of $\cN$ as shown in Fig.~\ref{cherry1}, and suppose that $T(\cN)=T(\cN')$. Let $p'_b$ denote the parent of $b$ in $\cN'$, and suppose that $p'_b$ is a tree vertex. Let $v'_1$ denote the child of $p'_b$ in $\cN'$ that is not $b$. If $v'_1$ is a tree vertex or a leaf, then either there is no phylogenetic $X$-tree displayed by $\cN'$ in which $\{a, b\}$ is a cherry or there is no phylogenetic $X$-tree displayed by $\cN'$ in which $\{b, V_q\}$ is a generalised cherry. This contradiction to Lemma~\ref{cherries2} implies that we may assume $v'_1$ is a reticulation.

\begin{sublemma}
Either $C_{v'_1}=\{a\}$ or $V_{v'_1}=V_q$.
\label{sets1}
\end{sublemma}

\begin{proof}
Using the arc $(p'_b, v'_1)$, it follows that $\cN'$ displays a phylogenetic $X$-tree in which $\{b, C_{v'_1}\}$ is a generalised cherry. Thus, as $T(\cN)=T(\cN')$, Lemma~\ref{cherries2} implies that if $a\in C_{v'_1}$, then $C_{v'_1}=\{a\}$. Furthermore, by the same lemma, if $a\not\in C_{v'_1}$, then, as $\cN'$ displays a phylogenetic $X$-tree in which $\{b, V_{v'_1}\}$ is a generalised cherry, $V_q\subseteq V_{v'_1}$. But $\cN$ displays a phylogenetic $X$-tree in which $\{b, V_q\}$ is a generalised cherry and so, as $T(\cN)=T(\cN')$, we also have $V_{v'_1}\subseteq V_q$. Hence if $a\not\in C_{v'_1}$, then $V_{v'_1}=V_q$.
\end{proof}

Let $q'$ denote the parent of $p'_b$ in $\cN'$.

\begin{sublemma}
\blue{The vertex $q'$ is a tree vertex.}
\label{treevertex1}
\end{sublemma}

\begin{proof}
\blue{Suppose that $q'$ is a reticulation, and} let $u'_1$ and $u'_2$ denote the parents of $q'$. If neither $(u'_1, q')$ nor $(u'_2, q')$ is a shortcut, then, for some $i\in \{1, 2\}$, there exists a leaf \blue{$\ell\in V_{u'_i}$} but $\ell\not\in V_q\cup \{a\}$. This implies that $\cN'$ displays a phylogenetic $X$-tree in which $\{b, Z\}$ is a generalised cherry with $\ell\in Z$, thereby contradicting Lemma~\ref{cherries2} as $T(\cN)=T(\cN')$. Thus, without loss of generality, we may assume that $(u'_2, q')$ is a shortcut.

Using the arc $(u'_1, q')$ but not $(p'_b, v'_1)$, the network $\cN'$ displays a phylogenetic $X$-tree in which \blue{$\{b, C_{u'_1}-b\}$} is a generalised cherry. Therefore, as $T(\cN)=T(\cN')$, it follows that if $a\in C_{u'_1}$, then \blue{$C_{u'_1}-b=\{a\}$}. Moreover, if $a\not\in C_{u'_1}$, then, as $\cN'$ displays a phylogenetic $X$-tree in which $\{b, V_{u'_1}\}$ is a generalised cherry, $V_q\subseteq V_{u'_1}$. But $\cN$ displays a phylogenetic $X$-tree in which $\{b, V_q\}$ is a generalised cherry and so, as $T(\cN)=T(\cN')$, we have $V_{u'_1}\subseteq V_q$. Hence if $a\not\in C_{u'_1}$, then $V_{u'_1}=V_q$.

If $(u'_2, u'_1)$ is an arc of $\cN'$, then \blue{$(u'_2, q')$} is a trivial shortcut. Therefore assume that $(u'_2, u'_1)$ is not an arc. If $C_{u'_1}=\{a\}$, then, using $(u'_2, q')$ and not $(p'_b, v'_1)$, it is easily seen that $\cN'$ displays a phylogenetic $X$-tree in which $\{b, Z\}$ is a generalised cherry, where $a\in Z$ and $|Z|\ge 2$. By (\ref{sets1}), this contradiction to Lemma~\ref{cherries2} implies that $C_{v'_1}=\{a\}$, and so $V_{u'_1}=V_q$.

Now, \blue{using $(u'_2, q')$ and avoiding $(p'_b, v'_1)$, we have that} $\cN'$ displays a phylogenetic $X$-tree in which \blue{$\{b, C_{u'_2}-\{a, b\}\}$} is a generalised cherry. \blue{So,} by Lemma~\ref{cherries2}, \blue{$C_{u'_2}-\{a, b\}\subseteq C_q-b$}. On the other hand, $\cN$ displays a phylogenetic $X$-tree $\cT$ in which $\{b, C_q-b\}$ is a generalised cherry. Since $V_{u'_1}=V_q$, for $\cN'$ to display $\cT$, we must have \blue{$C_q-b\subseteq C_{u'_2}-\{a, b\}$}. Thus \blue{$C_q-b=C_{u'_2}-\{a, b\}$}.

\blue{Let $P'$ be a path in $\cN'$ from $u'_2$ to $u'_1$. Since $(u'_2, u'_1)$ is not an arc, $P'$ contains at least one vertex, $w'$ say, in addition to $u'_2$ and $u'_1$. If the child vertex of $u'_2$ that is not $q'$ has the property that one of its children is not on $P'$, choose $w'$ to be this vertex. Otherwise, if this doesn't occur, then choose $w'$ to be a tree vertex on $P'$ that is not $u'_2$ which is the start of a tree-path to a leaf avoiding $u'_1$. It is easily checked that such a vertex exists. In either case, let $x'$ denote the child of $w'$ that does not lie on $P'$. Observe that, using $(w', x')$, there is a path from $u'_2$ to a leaf avoiding $u'_1$. Now using $(u'_1, q')$, $(w', x')$, and the arcs on $P'$, it is easily seen that} $\cN'$ displays a phylogenetic $X$-tree $\cT'$ with generalised cherry $\{b, C'_q\}$, where $V_q\subseteq C'_q\subset C_q$ and cluster $C_q$. But, by Lemma~\ref{nosplitting}, \blue{if $\mathcal S'$ is an embedding of $\cT'$ in $\cN$, then the vertex of $\mathcal S'$ corresponding to $C_q$ is $q$. In particular, $\cN$ does not display $\cT'$.} This completes the proof of \blue{(\ref{treevertex1})}.
\end{proof}

\blue{By} (\ref{treevertex1}), $q'$ is a tree vertex. Let $v'_2$ be the child of $q'$ that is not $p'_b$. Note that $v'_1\neq v'_2$; otherwise, $(q', v'_2)$ is a trivial shortcut. Using the arc $(q', v'_2)$ and not $(p'_b, v'_1)$, the network $\cN'$ displays a phylogenetic $X$-tree in which $\{b, C_{v'_2}\}$ is a generalised cherry. Therefore, by Lemma~\ref{cherries2}, if $a\in C_{v'_2}$, then $C_{v'_2}=\{a\}$. Furthermore, if $a\not\in C_{v'_2}$, then again using $(q', v'_2)$ and not $(p'_b, v'_1)$, we have that $\cN'$ displays a phylogenetic $X$-tree in which $\{b, V_{v'_2}\}$ is a generalised cherry. By Lemma~\ref{cherries2}, $V_q\subseteq V_{v'_2}$. But $\cN$ displays a phylogenetic $X$-tree in which $\{b, V_q\}$ is a generalised cherry and so, by Lemma~\ref{cherries2} again, we have $V_{v'_2}\subseteq V_q$. Thus if $a\not\in C_{v'_2}$, then $V_{v'_2}=V_q$. In combination with (\ref{sets1}), we now have

\begin{sublemma}
$\{V_{v'_1}, V_{v'_2}\}=\{\{a\}, V_q\}$. Furthermore, if $V_{v'_i}=\{a\}$, then $C_{v'_i}=\{a\}$ for each $i\in \{1, 2\}$.
\label{locate1}
\end{sublemma}

Using arcs $(p'_b, v'_1)$ and $(q', v'_2)$, it follows that $\cN'$ displays a phylogenetic $X$-tree $\cT'$ with generalised cherries $\{b, V_{v'_1}\}$ and $\{V_{v'_1}\cup b, V_{v'_2}\}$. Since $T(\cN)=T(\cN')$, we have that $\cN$ displays $\cT'$ as well. But then, by considering an embedding of $\cT'$ in $\cN$ together with (\ref{locate1}), it is easily seen that $\cN$, and therefore $\cN'$, displays a phylogenetic $X$-tree $\cT$ with generalised cherries $\{b, V_{v'_2}\}$ and $\{V_{v'_2}\cup b, V_{v'_1}\}$. To see this, observe that an embedding of $\cT$ in $\cN$ can be obtained from an embedding of $\cT'$ in $\cN$ by either deleting $(p_a, p_b)$ and adding $(q, p_b)$, or deleting $(q, p_b)$ and adding $(p_a, p_b)$. Let $u'_1$ be the parent of $v'_1$ that is not $p'_b$.

\begin{sublemma}
The arc $(u'_1, v'_1)$ is a shortcut in $\cN'$. In particular, $(u'_1, q')$ is an arc in $\cN'$.
\label{shortcut2}
\end{sublemma}

\begin{proof}
Consider an embedding $\mathcal S'$ of $\cT$ in $\cN'$. Clearly, $\mathcal S'$ uses $(u'_1, v'_1)$. If $(u'_1, v'_1)$ is not a shortcut, then $\cN'$ has a tree-path from $u'_1$ to a leaf that is not in $V_q\cup a$. But then $\mathcal S'$ is not an embedding of $\cT$ in $\cN'$. Thus $(u'_1, v'_1)$ is a shortcut in $\cN'$.

Now, in $\cN'$, there is a tree-path $P'$ from $u'_1$ to a leaf $\ell$. \blue{Since $\mathcal S'$ is an embedding of $\cT$ in $\cN'$, either $\ell=b$, or $v'_2$ is a tree vertex and $\ell$ is at the end of a tree-path for $v'_2$. Both possibilities imply that there is a unique path in $\cN'$ from $u'_1$ to $b$. Choose $P'$ to be this path and, as it is a tree-path, every vertex on $P'$, except $b$, is a tree vertex. Let $t'$ denote the parent of $q'$ and observe that $t'$ is on $P'$. Say $t'\neq u'_1$, and let $w'$ be the child of $t'$ that is not $q'$. If $w'=v'_2$, then $\cN'$ has a trivial shortcut, so $w'\neq v'_2$. It follows that there is a tree-path from $w'$ to a leaf $\ell'$ such that $\ell'\not\in V_q\cup a$. Using $(u'_1, v'_1)$, $(q', v'_2)$, $(t', w')$, and the arcs on $P'$, there is a phylogenetic $X$-tree $\cT'_1$ displayed by $\cN'$ with generalised cherries $\{b, V_{v'_2}\}$ and $\{V_{v'_2}\cup b, V_{w'}\}$. Note that $\ell'\in V_{w'}$. By considering an embedding of $\cT'_1$ in $\cN$, it is easily seen that $\cN$, and therefore $\cN'$ displays a phylogenetic $X$-tree $\cT_1$ with generalised cherries $\{b, V_{v'_1}\}$ and $\{V_{v'_2}, V_{w'}\}$. If $v'_2$ is a tree vertex in $\cN'$, then $\cN'$ does not display $\cT_1$. Therefore we may assume that $v'_2$ is a reticulation in $\cN'$.}

\blue{If $w'$ is not reachable from $v'_2$, then $\ell'\not\in C_q\cup a$, in which case, using $(u'_1, v'_1)$, $(t', w')$, the arcs on $P'$, but not $(q', v'_2)$, we have that $\cN$ displays a phylogenetic $X$-tree with a generalised cherry $\{b, Z\}$, where $\ell'\in Z$. This contradicts Lemma~\ref{cherries2}, and so $w'$ is reachable from $v'_2$. But then using $(u'_1, v'_1)$, $(t', w')$, the arcs on $P'$, but not $(q', v'_2)$, it follows that $\cN'$ displays a phylogenetic $X$-tree such that neither $\{a, b\}$ nor $\{b, C'_q\}$, where $V_q\subseteq C'_q$, is a generalised cherry. This last contradiction to Lemma~\ref{cherries2}, implies that $t'=u'_1$, that is $(u'_1, q')$ is an arc in $\cN'$.}
\end{proof}

Let $\mathcal S$ be an embedding of $\cT$ in $\cN$. Let $t$ denote the vertex in $\mathcal S$ corresponding to the last common ancestor of $V_q\cup \{a, b\}$, and let $P_a$ and $P_q$ denote the paths in $\mathcal S$ from $t$ to $p_a$ and $t$ to $q$, respectively. Note that we may assume $\mathcal S$ is chosen so that there is no other embedding of $\cT$ in $\cN$ in which the vertex corresponding to the last common ancestor of $V_q\cup \{a, b\}$ is strictly reachable from $t$ in $\cN$.

\begin{sublemma}
In $\cN$, the paths $P_a$ and $P_q$ consist of the arcs $(t, p_a)$ and $(t, q)$, respectively.
\label{path1}
\end{sublemma}

\begin{proof}
We begin by observing that, apart from $p_a$ and $q$, there is no vertex on either $P_a$ or $P_q$ which is the start of a tree-path to a leaf avoiding $p_a$ and $q$. Otherwise, $\mathcal S$ is not an embedding of $\cT$ in $\cN$. First consider $P_a$, and suppose that $(t, u)$ is an arc on $P_a$, where $u\neq p_a$. Assume $u$ is a tree vertex. Then it has a child vertex, $w$ say, that is not on either $P_a$ or $P_q$. To see this, if $u$ has both of its child vertices on $P_a$, then one of its children is a reticulation, and so there is a tree-path from $u$ to a leaf avoiding $p_a$ and $q$. Furthermore, if $u$ has a child vertex on $P_q$, then either $\cN$ has a trivial shortcut or there is a tree-path from a vertex on $P_q$ to a leaf avoiding $p_a$ and $q$. Now, there is a tree-path from $w$ to a leaf $\ell_w$ such that $\ell_w\not\in \{a, b\}\cup V_q$. By (\ref{locate1}), either $C_{v'_1}=\{a\}$ or $C_{v'_2}=\{a\}$. If $C_{v'_1}=\{a\}$, then, by using $(q, p_b)$, the arcs on $P_a$ and $P_q$, and $(u, w)$, it is easily checked that there is a phylogenetic $X$-tree displayed by $\cN$ that is not displayed by $\cN'$. Moreover, if $C_{v'_2}=\{a\}$, then, by using $(p_a, p_b)$, the arcs on $P_a$ and $P_q$, and $(u, w)$, it is again easily checked that there is a phylogenetic $X$-tree displayed by $\cN$ that is not displayed by $\cN'$. These contradictions imply that $u$ is not a tree vertex.

Now assume that $u$ is a reticulation. Let $s$ denote the parent of $u$ that is not $t$. Evidently, $s$ is not on $P_a$. If $s$ is on $P_q$, then there is an embedding of $\cT$ in $\cN$ in which the least common ancestor of $V_q\cup \{a, b\}$ is strictly reachable from $t$, contradicting the choice of $\mathcal S$. Thus $s$ is not on $P_q$. As $\cN$ is normal, $(s, u)$ is not a shortcut and so there is a tree-path from $s$ to a leaf $\ell_s$, where $\ell_s\not\in \{a, b\}\cup C_q$. \blue{Note that $\ell_s$ is not reachable from $q$; otherwise, $s$ is reachable from $q$ and so $(t, u)$ is a shortcut, but $\cN$ has no shortcuts.} Applying essentially the same argument to that when $u$ is a tree vertex, we again obtain a contradiction to $T(\cN)=T(\cN')$ and conclude that $P_a$ consists of the arc $(t, p_a)$.

Now consider $P_q$ and suppose that $(t, u)$ is an arc on $P_q$. If $u$ is a tree vertex, then there is a child vertex, $w$ say, of $u$ that is not on $P_q$, and so there is a tree-path from $u$ to a leaf $\ell_w$, where $\ell_w\not\in V_q\cup \{a, b\}$. If $C_{v'_1}=\{a\}$, then, by using $(q, p_b)$, the arcs on $P_q$, and $(u, w)$, it is easily seen that $\cN$ displays a phylogenetic $X$-tree that is not displayed by $\cN'$. Moreover, if $C_{v'_2}=\{a\}$, then, by using $(p_a, p_b)$, the arcs on $P_q$, and $(u, w)$, it is again easily see that there is a phylogenetic $X$-tree displayed by $\cN$ that is not displayed by $\cN'$. These contradictions imply that $u$ is not a tree vertex, and so we may assume that $u$ is a reticulation. Let $s$ denote the parent of $u$ that is not $t$. As $\cN$ is normal, $(s, u)$ is not a shortcut and there is a tree-path from $s$ to a leaf $\ell_s$, where $\ell_s\not\in C_q\cup \{a, b\}$. \blue{Note that $s$ is not reachable from $q$; otherwise, $\cN$ has a directed cycle.} Applying essentially the same argument to that when $u$ is a tree vertex, we conclude that $P_q$ consists of the arc $(t, q)$. This completes the proof of (\ref{path1}).
\end{proof}

We complete the proof of Proposition~\ref{structure3} by considering $v'_2$ in $\cN'$. First assume that $v'_2$ is a tree vertex. Then, as $T(\cN)=T(\cN')$, it follows that, for each $i\in \{1, 2\}$, if $V_q=V_{v'_i}$, we also have $C_q-b=C_{v'_i}$. In particular, in combination with (\ref{locate1}) we have the outcome shown in Fig.~\ref{noncherry}(a). Now assume that $v'_2$ is a reticulation. Let $u'_2$ denote the parent of $v'_2$ that is not $q'$. If $(u'_2, v'_2)$ is not a shortcut, then there is a tree-path from $u'_2$ to a leaf not in $\{a, b\}\cup C_q$, in which case, by using $(u'_2, v'_2)$, it is easily checked that $\cN'$ displays a phylogenetic $X$-tree not displayed by $\cN$; a contradiction. So $(u'_2, v'_2)$ is a shortcut. Noting that $u'_2$ is a tree vertex, let $w'$ denote the child vertex of $u'_2$ that is not $v'_2$. If $w'\neq u'_1$, then there is a child vertex of $w'$ that is the initial vertex of a tree-path to a leaf not in $\{a, b\}\cup C_q$. But then, by using $(u'_2, v'_2)$, we have that $\cN'$ displays a phylogenetic $X$-tree not displayed by $\cN$; a contradiction. Thus $w'=u'_1$, and so $(u'_2, u'_1)$ is an arc in $\cN'$. Furthermore, as $T(\cN)=T(\cN')$, it follows that if $C_{v'_1}=\{a\}$, then $C_{v'_2}=C_q$, while if $C_{v'_2}=\{a\}$, then $C_{v'_1}=C_q$. Thus we have the outcome shown in Fig.~\ref{noncherry}(b), thereby completing the proof of the proposition.
\end{proof}

\section{Recursion Lemmas}
\label{recur}

With the structural outcomes of Propositions~\ref{structure1}, \ref{structure2}, and~\ref{structure3} in hand, we next establish the three lemmas that will allow the algorithm to recurse correctly. The proof of the first lemma is straightforward and omitted.

\begin{lemma}
Let $\cN$ and $\cN'$ be normal and tree-child networks on $X$, respectively, and suppose that $\{a, b\}$ is a cherry of $\cN$ and $\cN'$. Then $T(\cN)=T(\cN')$ if and only if $T(\cN\backslash b)=T(\cN'\backslash b)$.
\label{algorithm1}
\end{lemma}

\begin{lemma}
Let $\cN$ and $\cN'$ be normal and tree-child networks on $X$, and suppose that $\{a, b\}$ is a reticulated cherry of $\cN$ and $\cN'$ as shown in Figs.~\ref{cherry1} and~\ref{cherry2}, respectively. Then $T(\cN)=T(\cN')$ if and only if $T(\cN\backslash (p_a, p_b))=T(\cN'\backslash (p'_a, p'_b))$.
\label{algorithm2}
\end{lemma}

\begin{proof}
First observe that $T(\cN)-T(\cN\backslash (p_a, p_b))$ (resp.\ $T(\cN')-T(\cN'\backslash (p'_a, p'_b))$) consists of precisely the phylogenetic $X$-trees displayed by $\cN$ (resp.\ $\cN'$) in which $\{a, b\}$ is a cherry. Thus if $T(\cN)=T(\cN')$, then $T(\cN\backslash (p_a, p_b))=T(\cN'\backslash (p'_a, p'_b))$. Suppose $T(\cN\backslash (p_a, p_b))=T(\cN'\backslash (p'_a, p'_b))$, and let $\cT$ be a phylogenetic $X$-tree displayed by $\cN$. If $\{a, b\}$ is not a cherry in $\cT$, then, by the observation, $\cN\backslash (p_a, p_b)$, and therefore $\cN'\backslash (p'_a, p'_b)$, displays $\cT$. This implies that $\cN'$ displays $\cT$. So assume $\{a, b\}$ is a cherry in $\cT$. Let $\mathcal S$ be an embedding of $\cT$ in $\cN$. Note that $\mathcal S$ must use the arc $(p_a, p_b)$. Let $\mathcal S_1$ be the embedding in $\cN$ of a phylogenetic $X$-tree $\cT_1$ obtained from $\mathcal S$ by deleting $(p_a, p_b)$ and adding $(q, p_b)$. Since $\{a, b\}$ is not a cherry of $\cT_1$, it follows that $\cN'$ displays $\cT_1$, that is, $\cN'$ has an embedding $\mathcal S'_1$ of $\cT_1$. Now, by replacing $(q'_2,p'_b)$ with $(p'_a, p'_b)$ in $\mathcal S'_1$, we have an embedding of $\cT$ in $\cN'$. Hence $\cN'$ displays $\cT$, and so $T(\cN)\subseteq T(\cN')$. Similarly, $T(\cN')\subseteq T(\cN)$. Thus $T(\cN)=T(\cN')$.
\end{proof}

\begin{lemma}
Let $\cN$ and $\cN'$ be normal and tree-child networks on $X$, respectively. Suppose that $\{a, b\}$ is a reticulated cherry of $\cN$ as shown in Fig.~\ref{cherry3}, while $\cN'$ has the structure local to leaves $a$ and $b$ as shown in either Fig.~\ref{noncherry}(a) or Fig.~\ref{noncherry}(b). 
\begin{enumerate}[{\rm (i)}]
\item If $C_{v'_1}=\{a\}$, then $T(\cN)=T(\cN')$ if and only if
$$T(\cN\backslash (p_a, p_b))=
\begin{cases}
T(\cN'\backslash (p'_b, v'_1)), & \mbox{$v'_2$ a tree vertex or a leaf;} \\
T(\cN'\backslash \{(p'_b, v'_1), (u'_2, v'_2)\}), & \mbox{otherwise.}
\end{cases}$$

\item If $C_{v'_2}=\{a\}$, then $T(\cN)=T(\cN')$ if and only if
$$T(\cN\backslash (p_a, p_b))=
\begin{cases}
T(\cN'\backslash (u'_1, v'_1)), & \mbox{$v'_2$ a tree vertex or a leaf;} \\
T(\cN'\backslash \{(u'_1, v'_1), (u'_2, v'_2)\}), & \mbox{otherwise.}
\end{cases}$$
\end{enumerate}
\label{algorithm3}
\end{lemma}

\begin{proof}
We shall prove (i). The proof of (ii) is similar and omitted. Suppose $C_{v'_1}=\{a\}$. For convenience, let $\cN_1$ denote $\cN\backslash (p_a, p_b)$. Furthermore, let $\cN'_1$ denote $\cN'\backslash (p'_b, v'_1)$ if $v'_2$ is a tree vertex or a leaf; otherwise, let $\cN'_1$ denote $\cN'\backslash \{(p'_b, v'_1), (u'_2, v'_2)\}$. We begin by observing that $T(\cN)-T(\cN_1)$ (resp.\ $T(\cN')-T(\cN'_1)$) consists of precisely the phylogenetic $X$-trees displayed by $\cN$ (resp.\ $\cN'$) in which $\{a, b\}$ is a cherry. Therefore if $T(\cN)=T(\cN')$, then $T(\cN_1)=T(\cN'_1)$.

For the converse, suppose that $T(\cN_1)=T(\cN'_1)$. Let $\cT$ be a phylogenetic $X$-tree displayed by $\cN$. If $\{a, b\}$ is not a cherry in $\cT$, then, by the observation, $\cN_1$, and therefore $\cN'_1$, displays $\cT$. It follows that $\cN'$ displays $\cT$. So assume $\{a, b\}$ is a cherry in $\cT$. Let $\mathcal S$ be an embedding of $\cT$ in $\cN$. Since $\{a, b\}$ is a cherry in $\cT$, the embedding $\mathcal S$ uses $(p_a, p_b)$. Let $\mathcal S_1$ denote the embedding in $\cN$ of a phylogenetic $X$-tree $\cT_1$ obtained from $\mathcal S$ by deleting $(p_a, p_b)$ and adding $(q, p_b)$. Since $\{a, b\}$ is not a cherry in $\cT_1$, it follows that $\cN'$ has an embedding $\mathcal S'_1$ of $\cT_1$. This embedding $\mathcal S'_1$ must use $(u'_1, v'_1)$. By replacing $(u'_1, v'_1)$ with $(p'_b, v'_1)$ in $\mathcal S'_1$, it is easily seen that we have an embedding of $\cT$ in $\cN'$. Hence $\cN'$ displays $\cT$ and so $T(\cN)\subseteq T(\cN')$.

Now let $\cT'$ be a phylogenetic $X$-tree displayed by $\cN'$. If $\{a, b\}$ is not a cherry, then, by the observation, $\cN'_1$, and therefore $\cN_1$, displays $\cT'$. So $\cN$ displays $\cT'$. Assume $\{a, b\}$ is a cherry in $\cT'$. Let $\mathcal S'$ be an embedding of $\cT'$ in $\cN'$. As $\{a, b\}$ is a cherry in $\cT'$ \blue{and as any embedding of $\cT'$ in $\cN'$ must use $(u'_1, q')$, $(q', p'_b)$, $(p'_b, b)$ and, if it exists, $(u'_2, u'_1)$, it is easily seen that we may choose $\mathcal S'$ so that it} uses $(p'_b, v'_1)$ and $(q', v'_2)$. Let $\mathcal S'_1$ be the embedding in $\cN'$ of a phylogenetic $X$-tree $\cT'_1$ obtained from $\mathcal S'$ by deleting $(p'_b, v'_1)$ and adding $(u'_1, v'_1)$. Since $\{a, b\}$ is not a cherry in $\cT'_1$, it follows that $\cN$ has an embedding $\mathcal S_1$ of $\cT'_1$. This embedding $\mathcal S_1$ must use $(q, p_b)$. By replacing $(q, p_b)$ with $(p_a, p_b)$ in $\mathcal S_1$, it is easily checked that we obtain an embedding of $\cT'$ in $\cN$. Thus $\cN$ displays $\cT'$, and so $T(\cN')\subseteq T(\cN)$. We conclude that $T(\cN)=T(\cN')$.
\end{proof}

\section{The Algorithm}

We now give a formal description of our algorithm {\sc SameDisplaySet} for deciding if $T(\cN)=T(\cN')$, where $\cN$ and $\cN'$ are normal and tree-child networks on $X$, respectively. Immediately after the description, we show that {\sc SameDisplaySet} works correctly and analyse its running time.

\noindent {\sc SameDisplaySet} \\
\noindent{\bf Input:} Normal and tree-child networks $\cN$ and $\cN'$ on $X$, respectively. \\
\noindent{\bf Output:} {\em No} if $T(N)\neq T(\cN')$, and {\em Yes} if $T(\cN)=T(\cN')$.

\begin{enumerate}[1.]
\item \label{initial} Delete all trivial shortcuts in $\cN'$, and denote the resulting normal and tree-child networks on $X$ as $\cN_0$ and $\cN'_0$, respectively.

\item Set $i=0$.

\item \label{find} If the leaf set of $\cN_i$ has size two, return {\em yes}. Else, find a cherry or a reticulated cherry, say $\{a, b\}$, of $\cN_i$.

\item \label{abcherry1} If $\{a, b\}$ is a cherry, then determine if $\{a, b\}$ is a cherry of $\cN'_i$.
\begin{enumerate}[(a)]
\item If no, then return {\em No}.

\item Else, delete $b$ and its incident arc in $\cN_i$ and $\cN'_i$, and denote the resulting normal and tree-child networks as $\cN_{i+1}$ and $\cN'_{i+1}$, respectively. Go to Step~\ref{increment}.
\end{enumerate}

\item \label{abcherry2} Else, $\{a, b\}$ is a reticulated cherry of $\cN_i$.
\begin{enumerate}[(a)]
\item If the parent of $b$ in $\cN'_i$ is a reticulation, then determine if, up to isomorphism, the structure in $\cN'_i$ local to $a$ and $b$ is as shown in Fig.~\ref{cherry2}.
\begin{enumerate}[(i)]
\item If no, then return {\em No}.

\item Else, delete $(p_a, p_b)$ and $(p'_a, p'_b)$ from $\cN_i$ and $\cN'_i$ to obtain $\cN_{i+1}$ and $\cN'_{i+1}$, respectively. Go to Step~\ref{increment}.
\end{enumerate}

\item Else, the parent of $b$ in $\cN'_i$ is a tree vertex. Determine if, up to isomorphism, the structures in $\cN_i$ and $\cN'_i$ local to $a$ and $b$ are as shown in Fig.~\ref{cherry3} and Fig.~\ref{noncherry}(a) or Fig.~\ref{noncherry}(b), respectively.
\begin{enumerate}[(i)]
\item If no, then return {\em No}.

\item Else, let $\cN_{i+1}$ denote the normal network obtained from $\cN_i$ by deleting $(p_a, p_b)$. Further, if $C_{v'_1}=\{a\}$, let $\cN'_{i+1}$ denote the tree-child network obtained from $\cN'_i$ by deleting $(p'_b, v'_1)$ as well as $(u'_2, v'_2)$ if it exists. Otherwise, if $C_{v'_2}=\{a\}$, let $\cN'_{i+1}$ denote the tree-child network obtained from $\cN'_i$ by deleting $(u'_1, v'_1)$ as well as $(u'_2, v'_2)$ if it exists. Go to Step~\ref{increment}.
\end{enumerate}
\end{enumerate}

\item \label{increment} Increase $i$ by $1$ and go back to Step~\ref{find}.
\end{enumerate}

Theorem~\ref{main} immediately follows from the next theorem.

\begin{theorem}
Let $\cN$ and $\cN'$ be normal and tree-child networks on $X$, respectively. Then {\sc SameDisplaySet} applied to $\cN$ and $\cN'$ correctly determines if $T(\cN)=T(\cN')$. Furthermore, {\sc SameDisplaySet} runs in time \blue{quadratic} in the size of $X$.
\label{atlast}
\end{theorem}

\begin{proof}
Ignoring the running time, by Lemma~\ref{shortshort}, we may assume that $\cN'$ has no trivial shortcuts. Therefore, as there is exactly one phylogenetic tree for when $|X|=2$, the fact that {\sc SameDisplaySet} correctly determines whether or not $T(\cN)=T(\cN')$ follows by combining Propositions~\ref{structure1}, \ref{structure2}, and~\ref{structure3} and Lemmas~\ref{algorithm1}, \ref{algorithm2}, and~\ref{algorithm3}. Thus to complete the proof of the theorem, it suffices to show that the running time of the algorithm is quadratic in the size of $X$.

Let $n=|X|$ and note that the total number of vertices in a tree-child network is linear in the size of $X$ (see~\cite{mcd15}). Thus both $\cN$ and $\cN'$ have at most $O(n)$ vertices in total. Now consider {\sc SameDisplaySet} applied to $\cN$ and $\cN'$. Step~\ref{initial} is a preprocessing step that considers, for each reticulation $v$ in $\cN'$, whether there is an arc joining the parents of $v$. Since this takes constant time for each reticulation, this step takes $O(n)$ time to complete. For iteration $i$, Step~\ref{find} finds a cherry or a reticulated cherry in $\cN_i$. Since $\cN_i$ is normal, one way to do this is to construct a maximal path that starts at the root of $\cN_i$ and ends at a tree vertex. The two leaves below this tree vertex, say $a$ and $b$, either form a cherry or a reticulated cherry in $\cN_i$. As the total number of vertices in $\cN_i$ is $O(n)$, this takes time $O(n)$. If $\{a, b\}$ is a cherry in $\cN_i$, then Step~\ref{abcherry1} determines whether or not $\{a, b\}$ is a cherry in $\cN'_i$ and, if so, deletes $b$ in both $\cN_i$ and $\cN'_i$. Therefore Step~\ref{abcherry1} takes constant time. On the other hand, if $\{a, b\}$ is a reticulated cherry in $\cN_i$, then Step~\ref{abcherry2} is called. Similar to Step~\ref{abcherry1}, this step considers the structure in $\cN_i$ and $\cN'_i$ local to $a$ and $b$, but is less straightforward. In terms of running time, the longest part of the step to complete is in determining the cluster and visibility sets of certain vertices. \blue{A single postorder transversal of each of $\cN_i$ and $\cN'_i$ can be used to determine all cluster sets of $\cN_i$ and $\cN'_i$. Since $\cN$ and $\cN'$ are both binary, the number of arcs in each is $O(n)$, so this takes time $O(n)$. Furthermore, to determine the visibility set of a vertex $u$ of $\cN_i$, we delete $u$ and its incident arcs, and check, for each leaf $\ell$, whether the resulting rooted acyclic directed graph, $D_i$ say, has a path from the root to $\ell$. That is, loosely speaking, we want to find the `cluster set', $X'$ say, of the root in $D_i$. It then follows that the visibility set of $u$ is $X_i-X'$, where $X_i$ is the leaf set of $\cN_i$. A single postorder transversal of $D_i$ is sufficient to determine $X'$, so this takes time $O(n)$. Similarly, the visibility set of a vertex in $\cN'_i$ can be found in this way. As we only need to find the visibility sets of three vertices in $\cN_i$ and $\cN'_i$, the total time to determine the necessary visibility sets is $O(n)$.} Thus the time to complete Step~\ref{abcherry2}, including the deletion of certain arcs, is $O(n)$. Hence, each iteration of {\sc SameDisplaySet} takes $O(n)$. Since each iteration deletes at least one vertex or arc in each of $\cN$ and $\cN'$, it follows that there are $O(n)$ iterations, and so the entire algorithm runs in time $O(n^2)$.
\end{proof}

\end{document}